\documentclass[12pt]{amsart}
\usepackage{amsfonts}
\usepackage{amsmath}
\usepackage{amssymb}
\usepackage[margin=1.2in]{geometry}
\usepackage{verbatim}
\newtheorem{theorem}{Theorem}[section]
\newtheorem{corollary}[theorem]{Corollary}
\newtheorem{lemma}[theorem]{Lemma}
\newtheorem{proposition}[theorem]{Proposition}

\theoremstyle{definition}

\newtheorem{remark}[theorem]{Remark}

\numberwithin{equation}{section}

\begin{document}

\title[Hal\'{a}sz's theorem for Beurling numbers]{Hal\'{a}sz's theorem for Beurling generalized numbers}

\author[G.~Debruyne]{Gregory Debruyne}
\thanks{G.~Debruyne gratefully acknowledges support by a Postdoctoral Fellowship of the Research Foundation--Flanders}

\author[F.~Maes]{Frederick Maes}

\author[J.~Vindas]{Jasson Vindas} 
\thanks{The work of J. Vindas was supported by the Research
  Foundation--Flanders, through the FWO-grant number 1510119N}

\address{Department of Mathematics: Analysis, Logic and Discrete Mathematics\\ Ghent University\\
  Krijgslaan 281\\ B 9000 Ghent\\ Belgium} 
  \email{gregory.debruyne@UGent.be}
  \email{frederick.maes@UGent.be}
  \email{jasson.vindas@UGent.be}


\begin{abstract}
We show that Hal\'{a}sz's theorem holds for Beurling numbers under the following two mild hypotheses on the generalized number system: existence of a positive density for the generalized integers and a Chebyshev upper bound for the generalized primes.
\end{abstract}

\maketitle

\section{Introduction}  \label{sec:intro}

Hal\'{a}sz's theorem \cite{Halasz1968} is a cornerstone in classical probabilistic number theory \cite{ElliottBookI,T2015}. This important result has been generalized
 by several authors \cite{L-R2001,zhang1987} to the context of abstract analytic number theory; the most general version so far being the one recently obtained by Zhang for Beurling numbers in \cite{zhang2018}.

Let $1 < p_1 \le p_2
\le \ldots\,$ be a Beurling generalized prime number system. Its associated set of generalized integers (cf. \cite{BD69,beurling1937, diamond-zhangbook,MV07}) is the multiplicative semigroup generated by 1 and the generalized primes, which we arrange in a non-decreasing sequence taking multiplicities into account, $n_0=1<n_1\leq n_2\leq n_3\leq \dots.$ Denote as $N$ and $\pi$ the counting functions of the generalized integers and primes. As in classical number theory, we consider the weighted prime counting functions
$$
\Pi(x)=\sum_{p_k^{\alpha_k}\leq x} 1/\alpha_k \quad \mbox{ and } \quad \psi(x)= \int_{1}
^{x}\log u\: \mathrm{d}\Pi(u).$$
Given a function of (local) bounded variation $G$, we denote its Mellin-Stieltjes transform as $\widehat{G}(s)=\int_{1^{-}}^{\infty}x^{-s} \mathrm{d}G(x)$ and we use the notation $s=\sigma+it$ for complex variables.

Zhang's version of Hal\'{a}sz's theorem reads as follows. His result generalizes \cite[Theorem 3.1]{D-D-V2018}, where the set of hypotheses \eqref{eqChebyshev}, \eqref{eqDensity}, and \eqref{L1condRestricted} were actually introduced.
\begin{theorem}[{Zhang \cite{zhang2018}}]
\label{Halasz th1} Suppose that the generalized number system satisfies a Chebyshev upper estimate
\begin{equation}
\label{eqChebyshev} \psi(x)\ll x,
\end{equation}
the generalized numbers have positive density
\begin{equation}
\label{eqDensity}
N(x)\sim a x
\end{equation}
for some $a>0$, and ($\sigma\to 1^{+}$)
\begin{equation}
\label{L1condRestricted} \int_{1}^{\infty} \frac{|N(x)-ax|}{x^{\sigma+1}}\:
\mathrm{d}x \ll (\sigma-1)^{-\beta}\, \qquad\mbox{for some } \beta\in [0,1/2).
\end{equation}
Let $g$ be a completely multiplicative function such that $|g(n_k)|\leq 1$ for each $n_k$ and set $G(x)=\sum_{n_k\leq x} g(n_k)$. Then,
\begin{equation*}
G(x)\sim c x 
\end{equation*}
if and only if
\begin{equation*}
\widehat{G}(s) = \frac{c}{s-1}+o\left(\frac{1}{\sigma-1}\right) 
\end{equation*}
 uniformly for $t$ on compact intervals.
\end{theorem}

The aim of this article is to considerably improve Theorem \ref{Halasz th1}. 
We shall show that it still holds if one removes the condition \eqref{L1condRestricted} from its hypotheses. In addition to hold under weaker assumptions, our results  are somewhat more general as they also involve  slowly varying functions in the asymptotic formulas and apply to multiplicative functions on non necessarily discrete number systems. We mention that our method here is inspired by  the treatment of Schwarz and Spilker from \cite{S-S1994} of the Daboussi-Indlekofer elementary proof \cite{Daboussi-I1992} of the classical Hal\'{a}sz theorem. 

Finally, it should be pointed out that our considerations yield the following improvement to \cite[Theorem 3.1]{D-D-V2018}, where $M$ is the sum function of the M\"{o}bius function of a generalized number system. 

\begin{corollary}
\label{Halasz c 2} The positive density condition  \eqref{eqDensity} 
and the Chebyshev upper bound \eqref{eqChebyshev} imply the estimate $M(x)=o(x)$.
\end{corollary}

\section{Main result and some consequences}
\label{Section Halasz theorem}

Let us start with our definition of the analog of a multiplicative function on a non necessarily discrete generalized number system. In a broader sense \cite{beurling1937,diamond-zhangbook}, a Beurling generalized number system is merely a pair of non-decreasing right continuous functions $N$ and $\Pi$ with $N(1)=1$ and $\Pi(1)=0$, both having support in $[1,\infty)$, and subject to the relation $\mathrm{d}N=\exp^{\ast}(\mathrm{d}\Pi)$, where the exponential is taken with respect to  the (multiplicative) convolution of measures \cite{diamond-zhangbook}. Since the hypotheses used in this article always guarantee convergence of the Mellin transforms, the latter becomes equivalent to the zeta function identity  
\[
\zeta(s) =\int^{\infty}_{1^{-}} x^{-s}\mathrm{d}N(x)= \exp\left(\int^{\infty}_{1}x^{-s}\mathrm{d}\Pi(x)\right).
\]

 We shall say that a (complex-valued) measure $\mathrm{d}G$ (supported on $[1,\infty)$) is \emph{arithmetic} (w.r.t. the number system under consideration) if it is absolutely continuous with respect to $\mathrm{d}N$. Furthermore, we call it \emph{multiplicative} if it can be written as $\mathrm{d}G=\exp^{\ast}\left(g\:\mathrm{d}\Pi\right)$ for some function $g$. Clearly, every multiplicative measure is arithmetic.

We can now state the main result of this article, its proof will be postponed to Section \ref{section proof of Halasz}.

\begin{theorem}
\label{Halasz th2} Suppose the number system satisfies the upper and lower logarithmic density conditions  
\begin{equation}
\label{eqH2.1} \int_{1^{-}}^{x}\frac{\mathrm{d}N(u)}{u}\asymp \log x.
\end{equation}

Let $\mathrm{d}G=\exp^{\ast}\left(g \:\mathrm{d}\Pi\right)$ be a multiplicative arithmetic measure with $g=g_1+g_2$ such that $|g_{1}(x)|\leq 1$, the bound $\int_{1}^{x} |g_1(u)|\log u\: \mathrm{d}\Pi(u)\ll x$ holds, and $\int_{1}^{\infty}x^{-1}|g_2(x)|\mathrm{d}\Pi(x)<\infty$.
Then, for real constants $c$, $\alpha$, and a slowly varying function $L(u)$ with $|L(u)|=1$, the relation
\begin{equation}
\label{eqH2.2}
\int_{1}^{x}\frac{G(u)}{u}\:\mathrm{d}u= \frac{c x^{1+i\alpha}}{(1+i\alpha)^{2}}L(\log x) +o(x)
\end{equation}
is satisfied if and only if
\begin{equation}
\label{eqH2.3}
\widehat{G}(s)=  \frac{c}{s-1-i\alpha} L\left(\frac{1}{\sigma-1}\right)+o\left(\frac{1}{\sigma-1}\right)
\end{equation}
  holds uniformly for $t$ in compact intervals.
\end{theorem}

The asymptotic relation \eqref{eqH2.2} could be differentiated via elementary familiar arguments (e.g. \cite[Section I.18, p.~37]{korevaarbook}) if $G$ satisfies additional Tauberian hypotheses. For example, if $g(u)\geq 0$, so that $G$ is non-decreasing, we must essentially have $\alpha=0$ and $L(u)=1$ in \eqref{eqH2.2}; one then deduces $G(x)\sim cx$. We might apply this to $G(x)=N(x)$ itself; the next corollary clarifies even more why the hypotheses on $N$ in Theorem \ref{Halasz th1} are redundant.

\begin{corollary}
\label{Halasz c 1} Assume that the Chebyshev upper bound \eqref{eqChebyshev} holds. Then, for $a>0$, $N(x)\sim ax$ holds if and only if 
\begin{equation}
\label{eqH2.4.1}  
\int_{1}^{\infty} \frac{|N(x)-ax|}{x^{\sigma+1}}\:
\mathrm{d}x =o\left(\frac{1} {\sigma-1}\right).
\end{equation}
These relations are also equivalent to 
\begin{equation}
\label{eqH2.4}\zeta(s)=\frac{a}{s-1}+ o\left(\frac{1}{\sigma-1}\right),
\end{equation}
uniformly for $t$ on compact intervals.
\end{corollary}
\begin{proof} The implications $N(x)\sim a x$ $\Rightarrow$ \eqref{eqH2.4.1} $\Rightarrow$ \eqref{eqH2.4} trivially hold unconditionally. Assume now that \eqref{eqH2.4} holds, then in particular $\zeta(\sigma)\sim a/(\sigma-1)$ and the Hardy-Littlewood-Karamata theorem yields logarithmic density, 
\begin{equation}
\label{log-density}
\int_{1^{-}}^{x}\frac{\mathrm{d}N(u)}{u}\sim a\log x.
\end{equation}
By Theorem \ref{Halasz th2}, we have
$
\int_{1}^{x}N(u)/u\: \mathrm{d}u\sim ax$. As explained above, using that $N$ is non-decreasing one concludes that $N(x)\sim ax$, as required.
\end{proof}

Furthermore,

\begin{theorem}
\label{Halasz th3} Assume $N$ has positive density  \eqref{eqDensity}. Let $\mathrm{d}G=\exp^{\ast}\left(g \:\mathrm{d}\Pi\right)$ be a multiplicative arithmetic measure with $g=g_1+g_2$ such that $|g_{1}(x)|\leq 1$, the bound $\int_{1}^{x} |g_1(u)|\log u\: \mathrm{d}\Pi(u)\ll x$ holds, and $\int_{1}^{\infty}x^{-1}|g_2(x)|\mathrm{d}\Pi(x)<\infty$.
Then, for real constants $c$, $\alpha$, and a slowly varying function $L(u)$ with $|L(u)|=1$, the asymptotic relation
\begin{equation}
\label{eqH2.5}
G(x)= \frac{c x^{1+i\alpha}}{1+i\alpha}L(\log x) +o(x)
\end{equation}
is satisfied if and only if \eqref{eqH2.3} holds uniformly for $t$ in compact intervals.
\end{theorem}
\begin{proof} 
The non-trivial implication is \eqref{eqH2.3} implies \eqref{eqH2.5}. Set $\mathrm{d}G_{i}=\exp^{\ast}\left( g_i \:\mathrm{d}\Pi\right)$ and in addition consider the convolution inverse of $\mathrm{d}G_{2}$, that is,  $\mathrm{d}F= \exp^{\ast}\left( -g_2 \:\mathrm{d}\Pi\right) $. Since $\widehat{F}(s)$ is absolutely convergent on $\Re e\: s=1$, a small computation shows that $(\widehat{F}(s)-\widehat{F}(1+i\alpha))/(s-1-i\alpha)=o(1/(\sigma-1))$ uniformly for $t$ on compacts. Thus, with the same uniformity, \eqref{eqH2.3} yields
$$
\widehat{G_{1}}(s)= \widehat{G}(s)\widehat{F}(s)= \frac{c_1}{s-1-i\alpha} L\left(\frac{1}{\sigma-1}\right)+o\left(\frac{1}{\sigma-1}\right),
$$
where $c_1=c/\widehat{G}_2(1+i\alpha)$. Applying Theorem \ref{Halasz th2} to $\mathrm{d}G_{1}$, we obtain
\begin{equation*}
\int_{1}^{x}\frac{G_1(u)}{u}\:\mathrm{d}u= \frac{c_1 x^{1+i\alpha}}{(1+i\alpha)^{2}}L(\log x) +o(x).
\end{equation*}
 Let us verify that $G_1(x)/x$ is slowly oscillating (in the sense of Schmidt, cf. \cite[Def.~I.16.1, p.~32]{korevaarbook}). 
 Due to our hypothesis on $G_1$ it is clear that $|\mathrm{d}G_1|\leq \mathrm{d}N$. Hence, if $\eta>1$,  
$$
\left|\frac{G_{1}(\eta x)}{\eta x}- \frac{G_{1}(x)}{ x}\right| \leq  \left(\eta-1\right) \left|\frac{G_{1}(\eta x)}{\eta x}\right| + \frac{N(\eta x)-N(x)}{ x}\ll \eta-1 + o_{\eta}(1),
$$
by \eqref{eqDensity}.  Since $L(\log x)$ is slowly varying, a standard elementary Tauberian argument gives
$$
G_1(x)= \frac{c_1 x^{1+i\alpha}}{1+i\alpha}L(\log x) +o(x).
$$
The asymptotic formula \eqref{eqH2.2} then follows from a variant of Wintner's mean-value theorem  (i.e., Lemma \ref{Wintner lemma}(i) below).
\end{proof}

Combining Theorem \ref{Halasz th3} with \cite[Lemma~3.6]{D-D-V2018}, we immediately obtain Corollary \ref{Halasz c 2} with $\mathrm{d}M$ the convolution inverse of $\mathrm{d}N$, namely, the measure $\mathrm{d}M=\exp^{\ast}\left(-\mathrm{d}\Pi\right)$ (for discrete number systems $M$ is then the sum function of the Beurling analog of the M\"{o}bius function). It is worth pointing out that the hypothesis $N(x)\sim ax$ cannot be omitted in Corollary \ref{Halasz c 2}, as shown by \cite[Examples 4.2 and 4.3]{D-D-V2018}. We also mention that one can construct examples of number systems for which $M(x)=o(x)$ and  \eqref{eqDensity}  hold for some $a>0$, but for which the Chebyshev bound \eqref{eqChebyshev} fails; see for instance Kahane's example \cite[Example 3.6]{D-VL1th}.

The ensuing version of the Hal\'{a}sz mean-value theorem holds true.

\begin{theorem}
\label{Halasz th4} Assume the positive density condition  \eqref{eqDensity} 
and the Chebyshev upper bound \eqref{eqChebyshev} and let $\mathrm{d}G=\exp^{\ast}\left(g \:\mathrm{d}\Pi\right)$ be a multiplicative arithmetic measure such that $g=g_1+g_2$ with $|g_{1}(x)|\leq 1$ and $\int_{1}^{\infty}x^{-1}|g_2(x)|\mathrm{d}\Pi(x)<\infty$.

If there is $\alpha\in\mathbb{R}$ such that 
\begin{equation}
\label{eqH2.6} \int_{1}^{\infty} \frac{1-\Re e\: (g(x)x^{-i\alpha})}{x} \: \mathrm{d}\Pi(x)
\end{equation}
converges, then 
\begin{equation}
\label{eqH2.7} G(x)= \frac{ x^{1+i\alpha}}{1+i\alpha}\exp\left(-\int_{1}^{x} \frac{1-g(u)u^{-i\alpha}}{u} \: \mathrm{d}\Pi(u)\right)
 +o(x).\end{equation}
 
 Otherwise, if there is no such $\alpha$, then $G$ has zero mean-value,
\begin{equation}
\label{eqH2.8} G(x)=o(x).
\end{equation} 
  
In either case, there are real constants $c$, $\alpha$, and a slowly varying function $L(u)$ with $|L(u)|=1$ such that 
\eqref{eqH2.5} holds.
\end{theorem}
\begin{proof} Using again Lemma \ref{Wintner lemma}(i), we may assume that $g_2=0$.
The result can then be deduced from Theorem \ref{Halasz th3} along the same lines of the proof of the corresponding Hal\'{a}sz mean-value theorem for the natural numbers given e.g. in Elliott's book \cite[Chapter 6]{ElliottBookI}. Therefore, we only give a brief sketch and leave most details to the reader. When \eqref{eqH2.6} diverges for every $\alpha$, the classical argument involving Dini's theorem yields $\widehat{G}(s)=o(1/(\sigma-1))$ (see e.g. \cite[Lemma~3.1]{zhang2018}, one just uses here $\zeta(\sigma)\ll 1/(\sigma-1)$), so that we obtain \eqref{eqH2.8} via Theorem \ref{Halasz th3} with $c=0$. 

In the case of convergence of \eqref{eqH2.6}, one may assume $\alpha=0$, because a simple integration by parts computation then yields the general result. We note that an adapted version of \cite[Lemma~6.8, p.~242]{ElliottBookI} holds in view of $\psi(x)\ll x$, while  \cite[Lemma~6.9, p.~243]{ElliottBookI} is valid because of  $N(x)\sim ax$ in the form \eqref{eqH2.4}. Hence, similarly as in \cite[pp.~245--246]{ElliottBookI}, one derives that \eqref{eqH2.3} holds for $t$ on compacts with $\alpha=0$ and a slowly varying function $L$ with modulus 1 satisfying 
$$
c L(\log x)=\exp\left( \int_{1}^{\infty}\frac{g(u)-1}{u} \exp\left(-\frac{\log u}{\log x}\right)\: \mathrm{d}\Pi(u)\right)+o(1).
$$
By Theorem \ref{Halasz th3}, it thus just remains to verify that the latter integral expression equals
$$
\int_{1}^{x} \frac{g(u)-1}{u}\: \mathrm{d}\Pi(u)+o(1).
$$
But this can also be established reasoning as in \cite[pp.~246--247]{ElliottBookI} with the aid of $\psi(x)\ll x$ and the simple bound
$$
\int_{1}^{x} \frac{\log u}{u}\:\mathrm{d}\Pi(u)= \frac{\psi(x)}{x}+\int_{1}^{x}\frac{\psi(u)}{u^2}\:\mathrm{d}u\ll \log x.
$$

\end{proof}

As a simple corollary, one also obtains Wirsing's mean-value theorem in this context. Of course, Corollary \ref{Halasz c 2} is also a consequence of it.

\begin{corollary}\label{Halasz c3} Suppose the positive density condition  \eqref{eqDensity} 
and the Chebyshev upper bound \eqref{eqChebyshev} hold. Let $\mathrm{d}G=\exp^{\ast}\left(g \:\mathrm{d}\Pi\right)$ be a real-valued multiplicative arithmetic measure such that $g=g_1+g_2$ with $|g_{1}(x)|\leq 1$ and $\int_{1}^{\infty}x^{-1}|g_2(x)|\mathrm{d}\Pi(x)<\infty$. Then,
\[
\lim_{x\to\infty} \frac{G(x)}{x}= \exp\left(-\int_{1}^{\infty} \frac{1-g(x)}{x} \: \mathrm{d}\Pi(x)\right),
\]
where the right-hand side is taken as zero when the integral diverges.
\end{corollary}
\begin{proof} Indeed, the convergent case directly follows from Theorem \ref{Halasz th4}. Assume thus that \eqref{eqH2.6} diverges for $\alpha=0$. If it also diverges for all other values of $\alpha$, we are done as well since \eqref{eqH2.8} holds. If \eqref{eqH2.6} converges for some $\alpha\neq 0$, then \eqref{eqH2.5} holds for some $c\in \mathbb{R}$ and $L$. We need to show that necessarily $c=0$. If $c$ were not zero, we would have
$$\lim_{x\to\infty}\frac{G(x e^{\frac{\pi}{2\alpha}})}{G(x)}=ie^{\frac{\pi}{2\alpha}};$$ but this limit must be real so that one must either have $c=0$ or that such an $\alpha$ does not exist. \end{proof}

We end this section with a remark concerning the case of discrete generalized number systems.

\begin{remark}\label{Halasz rm 1}
All the results from this section cover the particular instance of multiplicative functions on a discrete generalized number system satisfying $|f(n_k)|\leq 1$ for every generalized integer $n_k$, provided the generalized number system has a positive density and a Chebyshev upper bound holds for the generalized primes. 

Given a multiplicative function $f$, the associated multiplicative measure is $\mathrm{d}G= f\mathrm{d}N$. The functions $f$ and $g$ in the representation $f\:\mathrm{d}N=\exp^{\ast}(g\: \mathrm{d} \Pi)$ determine one another by their values on generalized prime powers linked by means of the relations
\begin{equation}
\label{eqH2.9} 1+\sum_{\nu=1}^{\infty} \frac{f(p_{k}^{\nu})}{p_k^{s\nu}}=\prod_{\nu=1}^{\infty}\exp\left(\frac{g(p_k^{\nu})}{\nu p_{k}^{\nu s}}\right),
\end{equation}
which are obtained by comparing factors corresponding to each generalized prime $p_k$ in $\sum_{k=0}^{\infty} n^{-s}_{k}f(n_k)=\exp\left(\int^{\infty}_{1}x^{-s}g(x)\:\mathrm{d}\Pi(x)\right)$ with its Euler product. Taylor expanding the exponential and multiplying out the right-hand side of \eqref{eqH2.9}, one readily deduces that
\begin{equation}
\label{eqH2.10}
  f(p_k^{\nu})=\sum_{\nu=1\cdot m_1+2\cdot m_2+ \dots+\nu \cdot m_\nu}\: \prod_{j=1}^{\nu} \frac{1}{m_j!}\left(\frac{g(p^j_k)}{j}\right)^{m_j}.
\end{equation}
In particular, $g(p_k)=f(p_k)$ for every generalized prime. The formula \eqref{eqH2.10} can be rewritten in terms of the (exponential) complete Bell polynomials (see e.g. \cite[p.~134]{comtet1974}, where the notation $Y_n=B_n$ is employed),
$$
f(p_k^{\nu})=\frac{1}{\nu!} B_{\nu}(0! g(p_k), 1! g(p_k^2), \dots, (\nu-1)! g(p_k^{\nu})).
$$
Conversely, taking logarithms in \eqref{eqH2.9} and using \cite[Theorem A, p.~140]{comtet1974},
\begin{equation}
\label{eqH2.11}
g(p^{\nu}_{k})= \sum_{j=1}^{\nu} (-1)^{j-1}\frac{ (j-1)!}{(\nu-1)!}B_{\nu,j}(1!f(p_k),2!f(p_k^2), \dots, (\nu-j+1)! f(p_k^{\nu-j+1})),
\end{equation}
where the
$B_{\nu,j}$ stand for the partial Bell polynomials. In particular, if $f$ is completely multiplicative, we have $g(p_{k}^{\nu})=f(p_{k}^{\nu})$ for each $k$ and $\nu\geq 1$. In view of $|f(n_k)|\leq 1$, we find using \eqref{eqH2.11} and \cite[Eq.~(3h), Theorem~B, p.~135]{comtet1974}
\begin{equation*}
|g(p_k^{\nu})|\leq \sum_{j=1}^{\nu}\frac{ (j-1)!}{(\nu-1)!}B_{\nu, j}(1!,2!, \dots, (\nu-j+1)!)= 2^{\nu}-1.
\end{equation*}

We further decompose $g=\tilde{g}+h$ with 
\[
\tilde{g}(p_{k}^{\nu})= \begin{cases}
g(p_k^{\nu}) & \mbox{if } p_{k}>2\\
0 & \mbox{otherwise}.
\end{cases}
\]
It is clear that the multiplicative arithmetic measure $\mathrm{d}\tilde{G}=\exp^{\ast}(\tilde{g}\:\mathrm{d}\Pi)$ satisfies the hypotheses we have been considering in this section. The Mellin transform of $\mathrm{d}H=\exp^{\ast}(h\:\mathrm{d}\Pi)$ is simply  the Euler product
\[ \widehat{H}(s)= \prod_{p_k\leq 2} \left(1+\sum_{\nu=1}^{\infty} \frac{f(p_{k}^{\nu})}{p_{k}^{\nu s}}\right),
\]
which is obviously absolutely convergent for $\sigma>0$. Using that $\widehat{H}(s)\ll 1$ and $\widehat{H}'(s)\ll1$ on the half-plane $\sigma\geq 1$, the proof of Theorem \ref{Halasz th2} we give in Section \ref{section proof of Halasz} can readily be adapted to obtain \eqref{eqH2.2}  for $G(x)=\sum_{n_k\leq x}f(n_k)$ from \eqref{eqH2.3}, $|f(n_k)|\leq 1$, positive density $N(x)\sim ax$, and the Chebyshev upper bound. On the other hand, the conditions $|f(n_k)|\leq 1$ and $N(x)\sim ax$ imply that $G(x)/x$ is slowly oscillating, so that Theorem \ref{Halasz th3} is valid in this case. Applying Theorem \ref{Halasz th4} to $\tilde{G}$ and then Lemma \ref{Wintner lemma}(i) to $\mathrm{d}G= \mathrm{d}H \ast \mathrm{d}\tilde{G}$, the Hal\'{a}sz mean-value theorem takes the form: If there is $\alpha\in\mathbb{R}$ such that
\[
\sum_{k=1}^{\infty}\frac{1- \Re e\: (p_{k}^{-i\alpha}f(p_k))}{p_k}
\]
converges, then 
\[
\frac{1}{x}\sum_{n_k\leq x} f(n_k)= \frac{x^{i\alpha}}{1+i\alpha}\prod_{p_k\leq x}\left(1-\frac{1}{p_k}\right) \left(1+\sum_{\nu=1}^{\infty} \frac{f(p_{k}^{\nu})}{p_{k}^{\nu(1+i\alpha)}}\right) + o(1);
\]
otherwise, $f$ has zero mean-value. Moreover, the assertion in Corollary \ref{Halasz c3} becomes: If in addition $f$ is real-valued, we always have
\[
\lim_{x\to \infty}\frac{1}{x}\sum_{n_k\leq x} f(n_k)= \prod_{k=1}^{\infty}\left(1-\frac{1}{p_k}\right) \left(1+\sum_{\nu=1}^{\infty} \frac{f(p_{k}^{\nu})}{p_{k}^{\nu }}\right).
\]
\end{remark}

\section{Auxiliary elementary estimates}
\label{Section elementary estimates}
We start with a key estimate based on Rankin's method (cf. \cite[Section II.3]{S-S1994}).

\begin{proposition}
\label{prop Rankin} Let $\mathrm{d}G=\exp^{\ast}\left(g\:\mathrm{d}\Pi\right)$ be a multiplicative arithmetic measure such that
\begin{equation}
\label{Rankin bound eq 1}
\int_{1}^{x} |g(u)|\log u \: \mathrm{d}\Pi(u)\ll x \log ^{\beta} x, 
\end{equation}
with $\beta\geq 0$. Then, 
\begin{equation}
\label{Rankin bound eq 2}
\frac{G(x)}{x}\ll \log ^{\beta-1} x \exp\left( \int_{1}^{x} \frac{|g(u)|}{ u} \: \mathrm{d}\Pi(u)\right) .
\end{equation}
\end{proposition}
\begin{proof}
We first estimate $\int_{1}^{x}\log u\: |\mathrm{d}G(u)|$. Note that the multiplication by $\log$ operator is a derivation on the convolution algebra of measures \cite[Section 2.8]{diamond-zhangbook}. We have $\log \cdot \mathrm{d}G= \mathrm{d}G\ast \left( g \cdot  \log\: \mathrm{d}\Pi\right)$ and so \eqref{Rankin bound eq 1} yields
$$
\int_{1}^{x}\log u\: |\mathrm{d}G(u)|\ll \int_{1^{-}}^{x} \frac{x}{u}\log^{\beta}\left(\frac{x}{u}\right)|\mathrm{d}G(u)|\ll x \log^{\beta}x \exp\left( \int_{1}^{x}\frac{|g(u)|}{ u} \: \mathrm{d}\Pi(u)\right),
$$
where we have used that multiplying by $1/u$ commutes with the exponential of measures. We now apply Rankin's trick,
\begin{align*}
G(x)&\ll \int_{1}^{\sqrt{x}} \log u \:|\mathrm{d}G(u)|+ \frac{2}{\log x} \int_{\sqrt{x}}^{x} \log u \:|\mathrm{d}G(u)|
\\
&
\ll (\sqrt {x} \: \log^{\beta}x + x \log^{\beta-1}x ) \exp\left( \int_{1}^{x}\frac{|g(u)|}{ u} \: \mathrm{d}\Pi(u)\right).
\end{align*}
\end{proof}
From here we deduce:
\begin{corollary}
\label{lemma Rankin}
Suppose that $\int_{1}^{x}u^{-1}\:\mathrm{d} \Pi(u)\leq \log \log x +O(1)$. If $\mathrm{d}G=\exp^{\ast}\left(g\:\mathrm{d}\Pi\right)$ is such that 
\begin{equation}
\label{Rankin eq 3}
\int_{1}^{x} |g(u)|\log u \: \mathrm{d}\Pi(u)\ll x,
\end{equation}
  then
 \begin{equation}
\label{Rankin eq 6}
\frac{G(x)}{x}\ll \exp\left(\int_{1}^{x}\frac{|g(u)|-1}{u}\: \mathrm{d}\Pi (u) \right),
\end{equation} 
 
\begin{equation}
\label{Rankin eq 4}\int_{\Re e\: s=\sigma} \left|\frac{\widehat{G}(s)}{s}\right|^{2}|\mathrm{d}s|\ll \int_{0}^{\infty} e^{-2 y (\sigma-1)}\exp\left(2\int_{1}^{e^y}\frac{|g(u)|-1}{u}\: \mathrm{d}\Pi (u) \right)\mathrm{d}y  
\end{equation}
as $\sigma\to 1^{+}$, and
\begin{equation}
\label{Rankin eq 5}
\int_{\Re e\: s=\sigma} \left|\frac{\widehat{G}'(s)}{s\widehat{G}(s)}\right|^{2}|\mathrm{d}s|\ll \frac{1}{\sigma-1}.
\end{equation}
\end{corollary}
\begin{proof}
Proposition \ref{prop Rankin} gives \eqref{Rankin eq 6}.
The bound \eqref{Rankin eq 4} then follows from the Plancherel identity because for fixed $\sigma$ the function 
$\widehat{G}(s)/s$ is the Fourier transform of $G(e^{y})e^{-\sigma y}.$
Next, $-\widehat{G}'(s)/\widehat{G}(s)$ is the Mellin transform of the measure $g(u) \log u \mathrm{d}\Pi(u)$, whose primitive is $O(x)$ by \eqref{Rankin eq 3}. So, the Plancherel theorem implies \eqref{Rankin eq 5}. 
\end{proof}

Let us point out that $\int_{1}^{x}u^{-1}\:\mathrm{d} \Pi(u)\leq \log \log x +O(1)$ implies upper logarithmic density $\int_{1}^{x}u^{-1}\:\mathrm{d} N(u)\ll \log x$. Moreover, the condition \eqref{eqH2.1} turns out to be equivalent to a weak form of Mertens' formula. 

\begin{lemma}\label{lemma weak Mertens} \eqref{eqH2.1} holds if and only if
\begin{equation}
\label{weak Mertens}
\int_{1}^{x}\frac{\mathrm{d} \Pi(u)}{u}= \log \log x +O(1).
\end{equation}
In addition, these relations are equivalent to $\zeta(\sigma)\asymp 1/(\sigma-1).$
\end{lemma}
\begin{proof} Combine \cite[Proposition 4.5 and Theorem 4.7]{diamond-zhangbook} for \eqref{eqH2.1} if and only if \eqref{weak Mertens}. The equivalence between \eqref{eqH2.1} and $\zeta(\sigma)\asymp 1/(\sigma-1)$ follows from \cite[Proposition 4.2, Proposition 4.8, and Corollary 4.10]{diamond-zhangbook}.
\end{proof}

We refer the reader to \cite{diamond-zhangbook,Pollack2013} for more information on Mertens type results for Beurling numbers.

 We shall also need the ensuing simple lemma. Note that part (i) is a version of Wintner's mean-value theorem.

\begin{lemma}
\label{Wintner lemma}
Let $A$ and $B$ be two functions of local bounded variation on $[1,\infty)$ such that $\int_{1^{-}}^{\infty} u^{-1}|\mathrm{d}A(u)|<\infty$. Consider $\mathrm{d}D=\mathrm{d}A\ast \mathrm{d}B$. Given a slowly varying function $\ell$ with $|\ell(u)|=1$
and $\alpha,b\in\mathbb{R}$, we have:
\begin{itemize}
\item [(i)] $\displaystyle B(x)= b x^{1+i\alpha} \ell(x)+o(x)$ implies $\displaystyle D(x)= \widehat{A}(1+i\alpha)bx^{1+i\alpha} \ell(x)+o(x)$.
\item [(ii)] $ \int_{1}^{x}u^{-1}B(u)\mathrm{d}u=b x^{1+i\alpha} \ell(x)+o(x) $ implies 
$$ \int_{1}^{x}\frac{D(u)}{u}\mathrm{d}u=\widehat{A}(1+i\alpha)bx^{1+i\alpha} \ell(x)+o(x).$$
\end{itemize}

\end{lemma}
\begin{proof}
For (i), we have
$$
\frac{1}{x^{1+i\alpha}\ell(x)}\int_{1^{-}}^{x} \mathrm{d}A\ast \mathrm{d}B= \int_{1^{-}}^{x} \frac{B(x/u)}{(x/u)^{1+i\alpha}\ell(x/u)}\frac{\ell(x/u)}{\ell(x)} \:\frac{\mathrm{d}A(u)}{u^{1+i\alpha}}\to b \int_{1^{-}}^{\infty}\frac{\mathrm{d}A(u)}{u^{1+i\alpha}}. 
$$
For (ii), we notice that 
$$\int_{1}^{x}\frac{B(u)}{u}\:\mathrm{d}u= \int_{1^{-}}^{x}\log(x/u)\mathrm{d}B(u)= \int_{1^{-}}^{x} \mathrm{d}H\ast \mathrm{d}B,
$$
with $\mathrm{d}H(u)= u^{-1}\chi_{[1,\infty)}(u)\mathrm{d}u$. So, 
$$\int_{1}^{x}\frac{D(u)}{u}\mathrm{d}u= \int_{1^{-}}^{x} \mathrm{d}H\ast \mathrm{d}B\ast \mathrm{d}A,$$
whence we conclude that part (ii) is a special case of part (i).
\end{proof}

Finally, we translate \eqref{eqH2.2} into another weighted average for $\mathrm{d}G$.

\begin{lemma}
\label{Halasz l 1}
Let $G$ be a function of local bounded variation on $[1,\infty)$ such that $\int_{1}^{x}u^{-1}G(u)\mathrm{d}u = o(x\log x)$.
Consider 
\begin{equation}
\label{Proof halasz eq 1} F(x)=\int_{1}^{x}\left(\int_{1}^{u}\log y\:
  \mathrm{d}G(y) \right)\frac{\mathrm{d}u}{u}.
\end{equation}
Then, for 
$\ell$ slowly varying with $|\ell(u)|=1$ and some constant $c\in\mathbb{R}$,
$$
\int_{1}^{x}\frac{G(u)}{u}\:\mathrm{d}u = cx\ell (x)+o(x) \qquad \mbox{if and only if } \qquad F(x)=c x\ell(x)\log x+o(x\log x).
$$
\end{lemma}
\begin{proof}
Integrating by parts, $ \int_{1}^{x}u^{-1}G(u)\:\mathrm{d}u \sim cx \ell(x) $ is equivalent to 
$$
\int_{1}^{x} \log u \frac{G(u)}{u}\mathrm{d}u\sim c x\ell(x)\log x.
$$
We also have
$$
F(x)= \int_{1}^{x} \log u \frac{G(u)}{u}\mathrm{d}u- \int_{1}^{x} \frac{1}{u}\int_{1}^{u} \frac{G(y)}{y} \mathrm{d}y \mathrm{d}u=  \int_{1}^{x} \log u \frac{G(u)}{u}\mathrm{d}u+ o(x\log x),
$$
whence the claim follows.
\end{proof}
\section{Proof of Theorem \ref{Halasz th2}}
\label{section proof of Halasz}
We start with some reductions. We only need to show that \eqref{eqH2.3} implies \eqref{eqH2.2}. The same reasoning employed at the beginning of the proof of Theorem \ref{Halasz th3} and Lemma \ref{Wintner lemma}(ii) allow us to assume without loss of generality that $g_2=0$. So, our hypotheses on $g$ are $|g(x)|\leq 1$ and \eqref{Rankin eq 3}.
 We may also assume that $\alpha=0$, namely, we are supposing that 
 \begin{equation}
\label{eqH4.1}
\widehat{G}(s)=  \frac{c}{s-1} L\left(\frac{1}{\sigma-1}\right)+o\left(\frac{1}{\sigma-1}\right),
\end{equation}
uniformly for $t$ on compact intervals. In view of Lemma \ref{lemma weak Mertens}, the bound \eqref{Rankin eq 6} from Corollary \ref{lemma Rankin} applies, so $G(x)\ll x$. Therefore, $G$ fulfills the conditions of Lemma \ref{Halasz l 1} and from now on we can restrict our attention to the function $F$ defined in \eqref{Proof halasz eq 1}. 

We should prove that
\begin{equation}
\label{eqH4.2}
F(x)=c x L(\log x)\log x+o(x\log x).
\end{equation}
The Mellin-Stieltjes transform of the function $F$ is 
$
\widehat{F}(s)=-\widehat{G}'(s)/s.
$
Given $x > e$, it is convenient to set $\sigma_{x}=1+1/\log x$.  By the
Perron inversion formula, we have

\begin{equation*}
\frac{F(x)}{x}=-\frac{1}{2\pi i}\int_{\Re e\: s=\sigma_{x}}
\frac{x^{s-1}\widehat{G}'(s)}{s^{2}}\, \mathrm{d}s.
\end{equation*}
Next, we take a large number $\lambda>1$, fixed for the while.  We
split the integral over the line $\{s:\Re e\: s=\sigma_{x}\}$ into three pieces, taken over 
\begin{align*}\Gamma_0&= \{\sigma_{x}+it:\: |t|\leq
\lambda/\log x \},\\
\Gamma_1&=  \{\sigma_{x}+it:\: \lambda/\log x<|t|\leq \lambda\},
\\
\Gamma_2&=  \{\sigma_{x}+it:\: \lambda<|t|\} .
\end{align*}

The integral over $\Gamma_0$ can easily be handled using the condition \eqref{eqH4.1} and the fact that $L$ is slowly varying; proceeding exactly as in \cite[p.~239]{ElliottBookI}, we obtain
$$
-\frac{1}{2\pi i}\int_{\Gamma_0}
\frac{x^{s-1}\widehat{G}'(s)}{s^{2}}\, \mathrm{d}s= cL(\log x)\log x+ O\left(\frac{\log x}{\lambda}\right).
$$

We now employ \eqref{Rankin eq 5} and the Cauchy-Schwarz inequality in order to get
$$
\int_{\Gamma_j}
\frac{x^{s-1}\widehat{G}'(s)}{s^{2}}\, \mathrm{d}s\ll \log^{1/2}x \left( \int_{\Gamma_{j}} \left|
\frac{\widehat{G}(s)}{s}\right|^{2} |\mathrm{d}s| \right)^{1/2}, \qquad j=1,2.
$$
It remains to estimate the latter two integrals. For the integral over the unbounded intervals $\Gamma_2$, we can apply \eqref{Rankin eq 4} to the multiplicative arithmetic measures 
$$\exp^{\ast}\left(u^{\pm i(\lambda+m)  }g(u)\mathrm{d}\Pi(u)\right);$$ 
hence,
\begin{align*}
&\int_{\Gamma_{2}} \left|
\frac{\widehat{G}(s)}{s}\right|^{2} |\mathrm{d}s|\leq \sum_{m=0}^{\infty}\frac{1}{1+(\lambda+m)^{2}}\left(\int_{-\lambda -m-1}^{-\lambda-m} +\int_{\lambda+m}^{\lambda+m+1}
\right)\left|\widehat{G}( \sigma_{x}+it)\right|^{2} \mathrm{d}t
\\
&
\leq 5 \sum_{m=0}^{\infty}\frac{1}{1+(\lambda+m)^{2}}\int_{\Re e\: s=\sigma_{x}}\left( \left|\frac{\widehat{G}(s+i(\lambda+m))}{s}\right|^{2}+ \left|\frac{\widehat{G}( s-i(1+\lambda+m))}{s}\right|^{2}\right) |\mathrm{d}s|
\\
&
\ll \frac{1}{\lambda}\int_{0}^{\infty} e^{-2(\sigma_{x} -1)y}\mathrm{d}y=  O\left(\frac{\log x}{\lambda}\right).
\end{align*}
On the other hand, using \eqref{eqH4.1}, 
$$
\int_{\Gamma_{1}} \left|
\frac{\widehat{G}(s)}{s}\right|^{2} |\mathrm{d}s|\leq \left(|c| \frac{\log x}{\lambda} +o_\lambda (\log x)\right)^{1/2} \int_{\Re e\:s=\sigma_{x}} \left|
\frac{(\widehat{G}(s))^{3/4}}{s}\right|^{2} |\mathrm{d}s|.
$$
To deal with the last integral we notice that $(\widehat{G}(s))^{3/4}$ is the Mellin transform of the arithmetic measure 
$\exp^{*}\left( 3g/4\: \mathrm{d}\Pi\right).$ Applying Plancherel's identity and \eqref{Rankin bound eq 2} to this measure, using the upper bound from \eqref{weak Mertens} and the hypothesis $|g(u)|\leq1$, we obtain
\begin{align*}
\int_{\Re e\:s=\sigma_{x}} \left|
\frac{(\widehat{G}(s))^{3/4}}{s}\right|^{2} |\mathrm{d}s|&\ll \int_{0}^{\infty} e^{-2 y (\sigma_{x}-1)} y^{-2}\exp\left(2\int_{1}^{e^y}\frac{(3/4)|g(u)|}{u}\: \mathrm{d}\Pi (u) \right)\mathrm{d}y
\\
& \ll \int_{0}^{\infty} e^{-2 y (\sigma_{x}-1)}y^{-1/2}\mathrm{d}y\ll \log^{1/2} x.
\end{align*}

Collecting all estimates, we arrive at 
$$
\frac{F(x)}{x\log xL(\log x)}- c \ll \frac{1}{\lambda^{1/4}} +o_{\lambda}(1).$$ 
Taking first the limit
superior as $x\to\infty$ and then $\lambda\to\infty$, we have shown \eqref{eqH4.2}. This establishes Theorem \ref{Halasz th2}.

\begin{remark}
It is worth pointing out that we have not used the lower bound from \eqref{weak Mertens} in this section. Thus, our proof above shows that Theorem \ref{Halasz th2} still holds true if the hypothesis \eqref{eqH2.1} is relaxed to $\int_{1}^{x}u^{-1}\mathrm{d}\Pi(u)\leq \log \log x +O(1)$.
\end{remark}

\end{document}